\documentclass[10pt]{amsart}

\usepackage{amsmath, amsthm, amsfonts}
\usepackage{amssymb}
\usepackage{latexsym}
\newcommand{\abs}[1]{\left\vert#1\right\vert}

\newcommand\dv{{\, \mathrm d v}}

\newcommand{\dd}{{\, \mathrm d}}

\newcommand{\td}[2]{\frac{\mathrm d #1}{\mathrm d #2}}

\newcommand{\sk}{\smallskip}
\newcommand{\mk}{\medskip}

\def\R{\mathbb{R}}

\newtheorem{thm}{Theorem}

\newtheorem{lem}[thm]{Lemma}

\theoremstyle{definition}

\theoremstyle{remark}
\newtheorem{rem}[thm]{Remark}

\def\signra{\bigskip \begin{center} {\sc Ricardo
      Alonso \par\vspace{3mm} Institute for Pure and Applied Mathematics 
     \par (IPAM) UCLA - (CAAM) Rice University \par 
    6100 Main Street -- MS 134,
    Houston, TX 77005-1892, USA\par \vspace{3mm} email:}
    \texttt{rja2@rice.edu} \end{center}}

\def\signjc{\bigskip \begin{center} {\sc Jos\'e
      A. Ca\~nizo\par\vspace{3mm} Centre for Mathematical
    Sciences\par University of Cambridge \par Cambridge CB3
    0WA, UK\par\vspace{3mm} email:}
    \texttt{j.a.canizo@dpmms.cam.ac.uk}\end{center}}

\def\signig{\bigskip \begin{center} {\sc Irene
      Gamba\par\vspace{3mm}
      Department of Mathematics\par The
    University of Texas at Austin\par Texas
    78712, USA\par\vspace{3mm} e-mail:}
    \texttt{gamba@math.utexas.edu} \end{center}}
    
\def\signcm{\bigskip \begin{center} {\sc Cl\'ement
      Mouhot\par\vspace{3mm} Centre for Mathematical
    Sciences\par University of Cambridge \par Cambridge CB3
    0WA, UK\par\vspace{3mm} email:}
    \tt{C.Mouhot@dpmms.cam.ac.uk} \end{center}}

\title[Exponential moments in the Boltzmann equation]{A new approach
  to the creation and propagation of exponential moments in the
  Boltzmann equation}

\author[R. Alonso, J. A. Ca\~nizo, I. Gamba, and C. Mouhot]{Ricardo Alonso 
  \and Jos\'e A. Ca\~nizo
  \and Irene M. Gamba
  \and Cl\'ement Mouhot}

\thanks{RA acknowledges the support from the NSF grant DMS-0439872,
  IPAM and CAAM. JAC was supported by the project MTM2011-27739-C04
  from DGI-MICINN (Spain) and 2009-SGR-345 from AGAUR-Generalitat de
  Catalunya. IMG has been partially funded by NSF grant DMS-1109625
  and by DMS FRG-0757450.  CM \& JAC acknowledge the support from the
  ERC grant MATKIT. Support from the Institute from Computational
  Engineering and Sciences at the University of Texas at Austin is
  also gratefully acknowledged.}

\begin{document}
\maketitle

\begin{abstract}
  We study the creation and propagation of exponential moments of
  solutions to the spatially homogeneous $d$-dimensional Boltzmann
  equation. In particular, when the collision kernel is of the form
  $|v-v_*|^\beta b(\cos(\theta))$ for $\beta \in (0,2]$ with
  $\cos(\theta)= |v-v_*|^{-1}(v-v_*)\cdot \sigma$ and $\sigma \in
  \mathbb{S}^{d-1}$, and assuming the classical cut-off condition $
  b(\cos(\theta))$ integrable in $\mathbb{S}^{d-1}$, we prove that there exists
  $a > 0$ such that moments with weight $\exp(a \min\{ t,1\}
  |v|^\beta)$ are finite for $t>0$, where $a$ only depends on the
  collision kernel and the initial mass and energy. We propose a novel
  method of proof based on a single differential inequality for the
  exponential moment with time-dependent coefficients.
\end{abstract}

\mk 


\sk 


\mk 

\tableofcontents

\section{Introduction}
\label{sec:intro}

We consider the spatially homogeneous Boltzmann equation in dimension
$d \geq 2$ with initial condition $f_0 \ge 0$, given by
\begin{equation}
  \label{eq:Boltzmann}
  \partial_t f = Q(f,f),  \quad f(t,\cdot) = f_0
\end{equation}
where $f = f(t,v) \ge 0$ is a non-negative function depending on time
$t \geq 0$ and velocity $v \in \R^d$, with $d \geq 2$. We will assume
throughout this paper that $f_0$ has finite mass and energy, i.e.,
\begin{equation*}
  \label{eq:finite-mass-energy}
  \left\| f_0 \right\|_{L^1(1+|v|^2)} := 
  \int_{\R^d} (1+|v|^2) f_0(v) \dv < +\infty.
\end{equation*}
For $p \in [1,+\infty]$, we denote by $L^p$ the Lebesgue spaces of
$p$-integrable real functions on $\R^d$, and the notation $L^p(w(v)
\dv)$ (or simply $L^p(w(v))$) denotes the $L^p$ space with weight
$w(v)$. The \emph{collision operator} $Q(f,f)$ is given by
\begin{equation*}
  \label{eq:Q}
  Q(f,f)(v) := \int_{\R^d \times \mathbb {\mathbb S}^{d-1}} B(\abs{v-v_*}, \cos
  \theta)   (f'_* f' - f_* f) \dd v_* \dd \sigma,
\end{equation*}
representing the total rate of binary interactions due to particles
taking the direction of $v$ due to collisions, minus those that were
knocked out from the $v$ direction. We follow the usual notation $f
\equiv f(v)$, $f_* \equiv f(v_*)$, $f' \equiv f(v')$, $f_*' \equiv
f(v_*')$.  The vectors $v',v_*'$, which denote the velocities after an
elastic collision of particles with velocities $v,v_*$, are given by
\begin{equation*}
  \label{eq:post-collision}
  v' := \frac{v+v_*}{2} + \frac{\abs{v-v_*}}{2} \sigma,
  \qquad
  v_*' := \frac{v+v_*}{2} - \frac{\abs{v-v_*}}{2} \sigma.
\end{equation*}
The variable $\theta$ denotes the angle between $v-v_*$ and $\sigma$,
where $\sigma$ is the unit vector in the direction of the
postcollisional relative velocity.
On the collision kernel $B$ we assume that for some $\beta \in (0,2]$
\begin{equation}
  \label{eq:B}
  B(\abs{v-v_*}, \cos \theta)
  = \abs{v-v_*}^\beta\, b(\cos \theta),
\end{equation}
with the following cut-off assumption:
\begin{equation}
  \label{eq:b-Lq}
  b \in L^1\left([-1,1], (1-z^2)^{\frac{d-3}{2}}\dd z\right).
\end{equation}
If we define $\tilde{b}(\sigma) := b(e_1 \cdot \sigma)$, with $e_1 \in
{\mathbb S}^{d-1}$ any fixed vector, then \eqref{eq:b-Lq} is equivalent to
$\tilde{b} \in L^1({\mathbb S}^{d-1})$, which can be easily seen by a
spherical change of coordinates.

Throughout the paper $f$ always represents a solution to equation
\eqref{eq:Boltzmann} on $[0,+\infty)$ (in the sense of, e.g.,
\cite{MR1697562}) and we always write, for $p \geq 0$ (not necessarily
an integer),
\begin{equation}
  \label{eq:def-m_p}
  m_p = m_p(t) := \int_{\R^d} f(t,v) \abs{v}^{p} \dv.
\end{equation}

\medskip

\noindent {\bf Main results.} It is known that moments of order $p >
2$ and exponential moments ($L^1$-exponentially weighted estimates)
with weight up to $\exp(a|v|^2)$ for some $a > 0$ are propagated by
equation \eqref{eq:Boltzmann}
\cite{MR1233644,citeulike:2703167,B88,Bobylev1996Moment,GPV09}; that
is, they are finite for all times $t > 0$ if they are initially
finite, however with a deterioration of the constant $a$. In
\cite{citeulike:2703167} it was proved that in fact equation
\eqref{eq:Boltzmann} with $\beta > 0$ instantaneously creates all
moments of orders $p > 2$, which then remain finite for all times $t >
0$. Here the assumption that $\beta > 0$ is necessary, since the
result is not true for Maxwell molecules for instance
\cite{IkTr:56}. Moreover, moments with exponential weight up to
$\exp(a |v|^{\beta/2})$ for some constant $a>0$ were also shown to be
instantaneously created in \cite{MR2264623,citeulike:2703083}. In all
these proofs it was crucial to assume that the angular function $b$ is
in $L^q\big([-1,1], (1-z^2)^{\frac{d-3}{2}}\dd z\big)$ for $q>1$ as
done in \cite{BGP04,GPV09,AL2010}. We also refer to the recent work
\cite{Lu-Mouhot} for moment production estimates in the so-called
non-cutoff case, in which proofs are based on the optimization of the
traditional inductive argument
\cite{B88,Bobylev1996Moment,GPV09,MR2264623,citeulike:2703083}.

We have several noticeable contributions in this paper.  Indeed, we
can extend the existing propagation and creation of
$L^1$-exponentially weighted estimates to include the classical
cut-off assumption $b \in L^1\big([-1,1], (1-z^2)^{\frac{d-3}{2}}\dd
z\big)$ without using the iterative methods developed in
\cite{BGP04,GPV09,AL2010}), and also we slightly relax the assumptions
on the initial data by requiring only finite mass and energy, and not
necessarily finite entropy as in previous works on creation of moments
\cite{citeulike:2703167}.  In addition, we improve the weights for the
creation of $L^1$-exponentially weighted moments, with a weight up to
$\exp(a |v|^{\beta})$ (hence removing the $1/2$ factor which was
present in \cite{MR2264623,citeulike:2703083}) for solutions with
finite mass and energy, assuming only an integrability condition on
$b$.  More specifically, Theorem \ref{thm:exp-moment-creation} gives
an explicit rate of appearance of exponential moments by showing that
the coefficient multiplying $|v|^\beta$ in the exponential weight can
be taken linearly growing in time.

The most important point is that we introduce a new method of proof
that not only does not need iterative arguments but also allows for
all these improvements. This approach is also used in Theorem
\ref{thm:exp-moment-propag} for the propagation of exponential
moments, and extends these results to classical cut-off assumptions on
the angular cross section $b$.

\bigskip

\begin{thm}[Creation of exponential moments]
  \label{thm:exp-moment-creation}
  Let $f$ be an energy-conserving solution to the homogeneous
  Boltzmann equation \eqref{eq:Boltzmann} on $[0,+\infty)$ with
  initial data $f^0 \in L^1(1+\abs{v}^2)$, and assume \eqref{eq:B} and
  \eqref{eq:b-Lq} with $\beta \in (0,2]$. Then there are some
  constants $C, a > 0$ (which depend only on $b$, $\beta$ and the
  initial mass and energy) such that
  \begin{equation*}
    \label{eq:exp-generation-clean}
    \int_{\R^d} f(t,v) \,
    \exp \big( a \min\{ t, 1\} \abs{v}^\beta \big) \dd v
    \leq
    C
    \quad
    \text{ for } t \geq 0.
  \end{equation*}
\end{thm}

  We remark that the existence and uniqueness of energy-conserving
  solutions with initial data $f^0 \in L^1(1+\abs{v}^2)$ was proved in
  \cite{MR1697562}.

  As mentioned above, our approach also provides a new proof of the
  property of \emph{propagation} of exponential moments
  \cite{GPV09,BGP04}. This is stated in the following theorem:

\begin{thm}[Propagation of exponential moments]
  \label{thm:exp-moment-propag}
  Let $f$ be an energy-conserving solution to the homogeneous
  Boltzmann equation \eqref{eq:Boltzmann} on $[0,+\infty)$ with
  initial data $f^0 \in L^1(1+\abs{v}^2)$, and assume \eqref{eq:B} and
  \eqref{eq:b-Lq} with $\beta \in (0,2]$. Assume moreover that the
  initial data satisfies for some $s\in [\beta, 2]$
  \begin{equation}
    \label{eq:hyp-exp-init}
    \int_{\R^d} f_0(v) \exp \big( a_0 \abs{v}^s \big) \dv \leq C_0.
  \end{equation}
  Then there are some constants $C, a > 0$ (which depend only on $b$,
  $\beta$ and the initial mass, energy and $a_0$, $C_0$ in
  \eqref{eq:hyp-exp-init}) such that
  \begin{equation*}
    \label{eq:exp-propagation-clean}
    \int_{\R^d} f(t,v) \,
    \exp \big( a \abs{v}^s \big) \dv \leq    C
    \quad
    \text{ for } t \geq 0.
  \end{equation*}
\end{thm}

We give in Section \ref{sec:direct} a novel argument for proving these
results which is based on a differential inequality for the
exponential moment itself, and the exploitation of a discrete
convolution-type estimate for the exponential moment of the gain part
of the collision operator. This avoids the intricate combination of
induction and maximum principle arguments in the previous proofs of
propagation \cite{GPV09,BGP04} and appearance
\cite{MR2264623,citeulike:2703083} of exponential moments. It also
clarifies the structure underlying these induction arguments. The
starting point of both these previous works and our approach is the
creation and propagation of polynomial moments in
\cite{MR1233644,citeulike:2703167} 
and the Povzner inequalities proved in
\cite{BGP04,AL2010}. We include a short appendix which gathers some of
the classical technical results used along the proofs.

\mk

\section{Refresher on the sharp Povzner Lemma}
\label{sec:prelim}

The following lemma reflects the angular averaging property of the
spherical integral acting on positive convex test functions evaluated
at the postcollisional velocities.  These estimates are crucial to be
able to control in a sharp form the moments of the gain operator by
estimates for lower bounds of the loss operator.  They were originally
introduced in \cite[Corollary 1]{BGP04} and further
developed in \cite[Lemma 3 and 4]{GPV09} and more recently
in \cite[Lemma 2.6]{AL2010}. We summarize these results as follows:

\begin{lem}[Sharp Povzner (angular averaging) Lemma]
  \label{lem:Povzner}
  Assume that $b: (-1,1) \to [0,\infty)$ satisfies \eqref{eq:b-Lq},
  and impose without loss of generality the following normalization
  condition
  \begin{equation}
    \label{eq:normalization}
    \int_{-1}^1 b(z) (1-z^2)^{\frac{d-3}{2}}\dd z
    = \frac{1}{|{\mathbb S}^{d-2}|},
  \end{equation}
  where $|{\mathbb S}^{d-2}|$ is the area of the $(d-2)$-dimensional
  unit sphere. Then for $p \geq 1$ it holds that
  \begin{equation}
    \label{eq:povzner}
    \int_{{\mathbb S}^{d-1}} \left(
      \abs{v'}^{2p} + \abs{v_*'}^{2p}
    \right)
    b(\cos \theta) \dd \sigma
    \leq
    \gamma_p
    \left(\abs{v}^2 + \abs{v_*}^2\right)^p
  \end{equation}
  where $\gamma_p > 0$ are constants such that $\gamma_1 = 1$, $p
  \mapsto \gamma_p$ is strictly decreasing and tends to $0$ as $p \to
  \infty$.
\end{lem}

\begin{rem}
  In the case when the symmetrization $z \mapsto b(z) + b(-z)$ of $b$
  is nondecreasing in $[0,1]$, these constants are controlled by
\begin{equation}
\label{gammap}
 \gamma_p  \le  \frac{1}{|{\mathbb S}^{d-2}|} 
\int_{-1}^{1} b(z) \left(\frac{1+z}2\right)^p (1-z^2)^{\frac{d-3}{2}}\dd z \, .
\end{equation}
\end{rem}

\begin{rem}
  In addition, when $b \in L^q([-1,1], (1-z^2)^{(d-3)/2}\,dz)$ with $q
  >1$, the decay of $\gamma_p$ can be estimated and shown to be
  polynomial: there exists a constant $C > 0$ such that
  \begin{equation*}
    \label{eq:gamma_p}
    \gamma_p \leq \min\left\{ 1, \frac{C}{p^{1/q'}} \right\}
    \qquad (p > 1),
  \end{equation*}
  with $q'$ the H\" older dual of $q$ (i.e., $1/q + 1/q' =
  1$). Furthermore, in the case $q = +\infty$, that is, for $b$
  bounded, it holds that
  \begin{equation*}
    \label{eq:gamma_p-infty}
    \gamma_p \leq \min\left\{ 1, \frac{16 \pi b^*}{p+1} \right\}
    \qquad (p > 1),
  \end{equation*}
  with $b^* := \max_{-1 \leq z \leq 1} b(z)$.
\end{rem}

Let us now state the key a priori estimate on the polynomial moments,
which shall be used in the sequel. For later reference, we define the
following quantity for any $s, p > 0$:
\begin{equation}
  \label{eq:def-S_p}
  S_{s,p} = S_{s,p}(t) := \sum_{k=1}^{k_p} {p \choose k}
  \left(
    m_{sk + \beta} \, m_{s(p -k)}
    + m_{sk} \, m_{s(p -k) + \beta}
  \right),
\end{equation}
with $k_p$ the integer part of $(p+1)/2$.

\begin{lem}[A priori estimate on the polynomial moments]
  \label{lem:estim-poly}
  For $s \in (0,2]$ and $p_0 > 2/s$, the following a priori inequality
  is true whenever all the terms make sense:
  \begin{equation}
    \label{eq:di1}
    \td{}{t} m_{sp}
    \leq
    2 \gamma_{sp/2} S_{s,p}
    - K_1 m_{sp+\beta} + K_2  m_{sp}
    \quad \text{ for } t \geq 0, \ p \ge p_0 > \frac{2}{s},
  \end{equation}
  with $S_{s,p}$ given by \eqref{eq:def-S_p} and constants 
  \begin{equation}
    \label{eq:di0.13}
    K_1 :=
    2 (1-\gamma_{sp_0/2}) C_\beta m_0  \quad \mbox{ and
      } \quad  K_2= 2 \, m_\beta
  \end{equation}
  with $C_\beta := \min\{1, 2^{1-\beta}\}$. 

Alternatively in the case $\beta \in (0,1]$, it is possible to get rid
of the second constant, and obtain
  \begin{equation} 
    \label{eq:di0.13.2} 
 K_1 :=
    2 (1-\gamma_{sp_0/2}) \bar C_\beta m_0  \quad \mbox{ and
      } \quad  K_2= 0
  \end{equation}
  for some constant $\bar C_\beta$ depending on $\beta$ and the
  initial data.

  In both cases, the constant $\gamma_{sp_0/2}$ depends 
  on the integrability of the
  angular function $b$ and on $p_0 > 2/s$.
\end{lem}

\begin{proof} 
  Using Lemma \ref{lem:Povzner} one obtains for any $p
  \geq 2/s$:
  \begin{multline}
    \label{eq:di0}
    \td{}{t} m_{sp}
    \\
    \leq \gamma_{sp/2} \int_{\R^d \times \R^d} f
    f_* \left( \left( |v|^2 + |v_*|^2 \right)^{\frac{s p}{2}} - |v|^{s
        p} - |v_*|^{s p} \right) |v-v_*|^\beta \dd v \dd v_*
    \\
    - 2 (1- \gamma_{sp/2}) \int_{\R^d \times \R^d} f f_* |v|^{s p}
    |v-v_*|^\beta \dd v \dd v_*.
  \end{multline}
  In order to estimate the right hand side of \eqref{eq:di0} we first
  focus on an upper bound for its positive term.  Since $0 < s/2 \leq
  1$, then
  \[
  \left( |v|^2 + |v_*|^2 \right)^{\frac{s p}{2}} \leq \left(
    |v|^s + |v_*|^s \right)^p \, .
  \]
  Hence, using Lemma~\ref{lem:app1} in the Appendix (a classical
  result taken from \cite[Lemma 2]{BGP04}) and the estimate
  $|v-v_*|^\beta \leq 2 |v|^\beta + 2 |v_*|^\beta$ we obtain that, for
  any $p \geq 1$, the first integral in \eqref{eq:di0} is controlled
  by
  \begin{multline}
    \label{eq:di0.11}
    \gamma_{sp/2} \int_{\R^d \times \R^d} f
    f_* \left( \left( |v|^2 + |v_*|^2 \right)^{\frac{s p}{2}} - |v|^{s
        p} - |v_*|^{s p} \right) |v-v_*|^\beta \dd v \dd v_*
    \\
    \leq
    2 \gamma_{sp/2} S_{s,p}.
 \end{multline}

 The estimate of the negative term in \eqref{eq:di0} requires a
 control from below.
 When $\beta\in(0,1]$
 it follows from Lemma~\ref{lem:app2} in the Appendix (taken from
 \cite[Lemma 2]{GPV09}) that the lower bound for the negative term in
 \eqref{eq:di0} satisfies
 \begin{equation} \label{eq:di0.12}
  2 (1- \gamma_{sp/2}) \int_{\R^d \times \R^d} f f_* |v|^{s p}
    |v-v_*|^\beta \dd v \dd v_* \geq \  2 \, \bar C_\beta \, (1-\gamma_{sp/2}) 
    m_0 m_{sp+\beta} 
 \end{equation}
 for some constant $\bar C_\beta$ related to $\beta$ and the initial
 data.  So estimate \eqref{eq:di1} follows with $K_1$ and $K_2$ as in
 \eqref{eq:di0.13.2}.

 In the general case $\beta\in(0,2]$, the previous argument does not
 necessarily follow, yet it is still possible to obtain an easier
 lower bound that still allows for the control of moments and their
 summability. We use the fact that $|v-v_*|^\beta \geq 2^{1-\beta}
 |v|^\beta - |v_*|^\beta $ (which can be obtained from the triangle
 inequality and the inequality $(x+y)^\beta \leq C_\beta^{-1}(x^\beta
 + y^\beta)$ for $x,y \geq 0$.) This gives a lower bound for the
 negative term in \eqref{eq:di0}:
\begin{multline}
  \label{eq:di0.12.1}
  2 (1- \gamma_{sp/2}) \int_{\R^d \times \R^d} f f_* |v|^{s p}
  |v-v_*|^\beta \dd v \dd v_*
  \\ 
  \qquad\geq \  2 (1-\gamma_{sp/2}) 
  C_\beta m_0 m_{sp+\beta} - 2 m_\beta m_{sp}.
\end{multline}
Since $\gamma_{sp}$ decreases as $p \to \infty$, it follows that $2
(1-\gamma_{sp/2}) C_\beta m_0 \geq K_1$ for any $p \geq p_0$. Hence,
estimate \eqref{eq:di1} follows with $K_1$ and $K_2$ as in
\eqref{eq:di0.13}.
\end{proof}

\begin{rem} We note that neither in the work \cite{GPV09} nor in here
  the finiteness of the entropy is required, however it was needed in
  the earlier work \cite{citeulike:2703167} in order to obtain lower
  bounds for the negative term in \eqref{eq:di0}. If the solution has
  a finite entropy, then these lower bounds may be obtained by the
  same technique as in \cite{citeulike:2703167}. Observe however that
  the constant $\bar C_\beta$ in the case $\beta \in (0,1]$ with
  $K_2=0$ depends on the initial data in a non-trivial way, through
  the positive constant $C>0$ such that
  \begin{equation*}
    \int_{\R^d} f_0(v_*) \, |v-v_*|^\beta \dd v_* \ge C (1+ |v|^\beta)
  \end{equation*}
  which cannot be expressed simply in terms of the mass and energy of
  $f_0$. Nevertheless the general argument (involving $K_2>0$) does
  provide constants only depending on the initial data through its
  mass and momentum.
\end{rem}

Next, we recall and prove a very similar result to that in
\cite[Theorem 4.2]{citeulike:2703167}. The main difference is that
finiteness of the entropy of the initial condition is not required
here.

\begin{lem}[Creation and propagation of polynomial moments]
  \label{lem:moment-creation}
  Assume \eqref{eq:B} and \eqref{eq:b-Lq} with $0 < \beta \leq 2$. Set
  $s \in (0,2]$, and let $f$ be an energy-conserving solution to the
  homogeneous Boltzmann equation \eqref{eq:Boltzmann} with initial
  data $f_0 \in L^1(1+\abs{v}^2)$. For every $p > 0$ there exists a
  constant $C_{sp} \geq 0$ depending only on $p$, $s$, $b$ and the
  initial mass and energy, such that
  \begin{equation}
    \label{eq:moment-creation}
    m_{sp}(t) \leq C_{sp} \max \{ 1, t^{-sp/\beta} \}
    \quad \text{ for } t >0.
  \end{equation}
  If $m_{sp}(0)$ is finite, then the control can be improved to
  simply
  \begin{equation}
    \label{eq:moment-propag}
    m_{sp}(t) \leq C_{sp} \quad \text{ for } t \ge 0
  \end{equation}
  for some constant $C_{sp}$ depending only on $p$, $s$, $b$, the
  initial mass and energy, and $m_p(0)$.
\end{lem}

\begin{proof}
  Following a common procedure (see
  \cite{MR1697562,citeulike:2703167}), the estimates can be carried
  first on a truncated solution (for which all moments are finite and
  our calculations are rigorously justified), and then proved for the
  solution to the full problem by relaxing the truncation parameter.

  Let us prove \eqref{eq:moment-creation}: observe that by H\"older's
  inequality
  \[
  S_{s,p} \le C m_\beta m_{sp}
  \quad \text{ and }
  \quad m_{sp+\beta} \ge K m_{sp} ^{1+\beta/(sp)}
  \]
  for some constants $C$, $K>0$ depending only on $s$, $p$, the initial mass and energy. Since $\beta \le 2$, we have $1 \le
  2/\beta$ and therefore $m_\beta$ is controlled by the mass and
  energy. We deduce that $m_{sp}$ satisfies the differential
  inequality
  \begin{equation}\label{ee1}
  \td{}{t} m_{sp} \le C' m_{sp} - K m_{sp} ^{1+\beta/(sp)}
  \end{equation}
  for some other constant $C'>0$ depending only on $s$, $p$, the initial mass and energy. This readily implies the bound
  \eqref{eq:moment-creation} by computing an upper solution to this
  differential inequality, and the bound \eqref{eq:moment-propag} by a
  maximum principle argument.
\end{proof}

\begin{rem}
  Observe that the polynomial bound $O(t^{-sp/\beta})$ on the
  appearance of $m_p$ is \emph{not} optimal, as can be seen from
  \cite[Theorem~1.1]{MR1697562}. However our rate of appearance of
  exponential moments can be seen to be optimal by inspection of the
  simpler equation $\partial_t f = - C \left(1+|v|^\beta\right) f$
  which provides subsolutions to the Boltzmann
  equation. 
\end{rem}

\section{Proof of the main theorems}
\label{sec:direct}

In this section we give a proof of Theorems
\ref{thm:exp-moment-creation} and \ref{thm:exp-moment-propag} 
valid for any 
integrable cross-section $b$. We first carry out the estimates on a
finite sum of polynomial moments, and then pass to the
limit.

Our goal is to estimate the quantity
\begin{equation*}
  E_s(t,z)
  :=
  \int_{\R^d} f(t,v) \exp\big( z \abs{v}^s \big) \,\dd v
  =
  \sum_{p=0}^\infty m_{sp}(t) \frac{z^p}{p!}
\end{equation*}
where $s=\beta$ and $z=at$ for Theorem~\ref{thm:exp-moment-creation}
and $s \in (0,2]$ and $z=a$ for Theorem~\ref{thm:exp-moment-propag},
for some $a > 0$. For use below let us define the truncated sum as
\begin{equation*}
  E^n _s(t,z)  :=   \sum_{p=0}^n m_{sp}(t) \frac{z^p}{p!}
\end{equation*}
for $n \in \mathbb N$, $z \ge 0$, and $t \geq 0$. We also define
\begin{equation*}
  I^n _{s,\beta}(t,z) := \sum_{p=0}^n m_{sp+\beta}(t) \frac{z^{p}}{p!}.
\end{equation*}

Let us first prove the key lemma, which identifies the discrete
convolution structure.  This result gives a control for finite sums of
the moments associated to the gain operator. It is uniform in
$\beta\in (0,2]$:
\begin{lem}
  \label{lem:convolution}
  Assume $0 < \beta \le s \le 2$. For any $p_0 \ge 2/s$, we have the
  following functional inequality
  \begin{equation}
    \label{eq:convolution}
    \sum_{p=p_0} ^n \frac{z^p}{p!} S_{s,p}(t) \le 2 E^n _s (t,z) I^n _{s,\beta}(t,z)
  \end{equation}
where $S_{s,p}$ was defined in \eqref{eq:def-S_p}. 
\end{lem}

\begin{proof}
  Let us recall the definition of $S_{s,p}$ from \eqref{eq:def-S_p}:
  \begin{equation*}
    S_{s,p}
    := \sum_{k=1}^{k_p} {p \choose k}
    \left( m_{sk + \beta} \, m_{s(p - k)} + m_{sk} \, m_{s(p -k) +
        \beta} \right),
  \end{equation*}
  where $k_p$ is the integer part of $(p+1)/2$. The first part of the
  sum in the left hand side of \eqref{eq:convolution} can be bounded
  as:
\begin{multline*}
  \label{eq:sum-S_p-2}
  \sum_{p=p_0}^n \frac{z^p}{p!} \sum_{k=1}^{k_p} {p \choose k}
  m_{sk + \beta} \, m_{s(p - k)}
  = \sum_{p=p_0}^n \sum_{k=1}^{k_p} m_{sk + \beta}
  \frac{z^{k}}{k!}  \, m_{s(p - k)} \frac{z^{p-k}}{(p-k)!}
  \\
  \leq \sum_{k=1}^n m_{sk + \beta}
  \frac{z^{k}}{k!}  \sum_{p=\max\{p_0, 2k-1\}}^n m_{s(p - k)}
  \frac{z^{p-k}}{(p-k)!}  \leq I^n _{s,\beta} (t,z) E^n _s(t,z).
\end{multline*}
We carry out a similar estimate for the other part:
\begin{multline*}
  \sum_{p=p_0}^n \frac{z^p}{p!} 
  \sum_{k=1}^{k_p} {p \choose k}
  m_{sk} \, m_{s(p-k)  + \beta}
  =
  \sum_{p=p_0}^n 
  \sum_{k=1}^{k_p}
  m_{sk} \frac{z^{k}}{k!}
  \, m_{s(p- k) + \beta} \frac{z^{p-k}}{(p-k)!}
  \\
  \leq
  \sum_{k=1}^n
  m_{sk} \frac{z^{k}}{k!}
  \sum_{p=\max\{p_0, 2k-1\}}^n
  m_{s(p- k)+\beta} \frac{z^{p-k}}{(p-k)!}
  \leq
  E^n_s(t,z) I^n_{s,\beta}(t,z)
\end{multline*}
which concludes the proof.
\end{proof}

\bigskip We now can prove both Theorem~\ref{thm:exp-moment-creation}
and Theorem~\ref{thm:exp-moment-propag}.  We write the proof first for
the case $\beta \in (0,1]$ with the choice of constants
\eqref{eq:di0.13.2} in \eqref{eq:di1} (hence with $K_2=0$). Later we
show the corresponding estimates for the full range $\beta\in (0,2]$
using the choice of constants \eqref{eq:di0.13} in \eqref{eq:di1}.
  
\begin{proof}[Proof of Theorem~\ref{thm:exp-moment-creation}]
  First we notice that it is enough to prove the following (under the
  same assumptions): there are some constants $T, C, a > 0$ (which
  depend only on $b$ and the initial mass and energy) such that
  \begin{equation}
    \label{eq:exp-generation}
    \int_{\R^d} f(t,v) \, \exp \big( a t \abs{v}^\beta \big) \dd v
    \leq
    C
    \quad
    \text{ for } t \in [0, T].
  \end{equation}
  Indeed, since the assumptions of lower and upper bounds on the mass
  and energy are satisfied uniformly in time along the flow, for $t
  \ge T$ it is then possible to apply \eqref{eq:exp-generation}
  starting at time $(t-T)$.

  Hence, we aim at proving the estimate \eqref{eq:exp-generation}. We
  set $s = \beta$. Consider $a >0$ to be fixed later, $n \in \mathbb
  N$ and define $T > 0$ as
\begin{equation*}
  \label{eq:def-T}
  T := \min\Big\{ 1 \ ; \ \sup\big\{t > 0 \ \text{ s.t. } \ E^n_\beta (t,at) < 4 m_0 \big\}\Big\}.
\end{equation*}
The definition is consistent since $E^n _\beta(0,0) = m_0$ and the
Lemma~\ref{lem:moment-creation} ensures that $T>0$ for each given
$n$. The bound of $1$ is not essential, and is included just to ensure
that $T$ is finite. We note that a priori such $T$ depends on the
index $n$ in the sum $E^n_\beta$. However, we will show that $T$ has a lower
bound that depends only on $b$, $\beta$ and the initial mass and
energy, thus proving the theorem. Unless otherwise noted, all
equations below which depend on time are valid for $t \in [0,T]$.

Choose an integer $p_0 > 2/\beta$, to be fixed later. Starting from
Lemma~\ref{lem:estim-poly} (inequality~\eqref{eq:di1}), we have
\begin{equation}
  \label{eq:di4.5}
  \td{}{t} m_{\beta p}
  \leq
  2 \gamma_{\beta p/2} S_{\beta,p}
  - K_1 m_{\beta(p+1)} 
  \quad \text{ for } t \geq 0, \ p \geq p_0,
\end{equation}
with $S_{\beta, p}$ given by \eqref{eq:def-S_p} and $K_1$ defined in 
\eqref{eq:di0.13.2}, independent of $p$ with $p\geq p_0$ as soon as
$p_0$ is strictly bounded away from $2/\beta$. 

In addition, from Lemma \ref{lem:moment-creation}
(inequality~\eqref{eq:moment-creation}) we know that there exists a
constant $C_{p_0} > 0$ (depending on $p_0$) such that
\begin{equation}
  \label{eq:small-moments}
  \sum_{p=0}^{p_0} m_{\beta p}(t)\,t^p \leq C_{p_0}
  \quad \text{ for all } t \in [0,T].
\end{equation}
%
Taking any $a < 1$ and using the product rule,
\begin{multline*}
  \td{}{t} \sum_{p=p_0}^n m_{\beta p} \frac{(at)^p}{p!}
  \\
  \leq \sum_{p=p_0}^n \frac{(at)^p}{p!}  \left( 2
    \gamma_{\beta p/2} S_{\beta,p} - K_1 m_{\beta( p+1)} \right)
+ a
  \sum_{p=p_0}^n m_{\beta p} \, \frac{(at)^{p-1}}{(p-1)!}
  \\
  \leq 2 \sum_{p=p_0}^n \frac{(at)^p}{p!} \gamma_{\beta p/2}
  S_{\beta, p} +
  (a - K_1) I^n _{\beta,\beta}(t,at) 
+ (K_1 +a)
  \sum_{p=1}^{p_0} m_{\beta p}
  \,\frac{(at)^{p-1}}{(p-1)!}
  \\
  \leq 2 \sum_{p=p_0}^n \frac{(at)^p}{p!} \gamma_{\beta p/2} S_{\beta,p} +
  (a- K_1) I^n_{\beta,\beta}(t,at) 
+ \frac1t (K_1+a) C_{p_0},
\end{multline*}
where we have used $a < 1$ and \eqref{eq:small-moments} in
the last step. Hence, from Lemma~\ref{lem:convolution}
(inequality~\eqref{eq:convolution}) we obtain
\begin{equation*}
  \td{}{t} \sum_{p=p_0}^n m_{\beta p} \frac{(at)^p}{p!}
   \leq I^n_{\beta,\beta}(t,at) \Big[
    4 \gamma_{\beta p_0/2} E^n_\beta(t,at) + (a- K_1)
  \Big] 
+ \frac1t (K_1+a) C_{p_0}.
 \end{equation*}
 Next, choose $p_0$ large enough such that $16 \gamma_{\beta p_0/2}
 m_0 \leq (1/4) K_1$ (or equivalently, by using the definition of
 $K_1$ in \eqref{eq:di0.13}, $\gamma_{\beta p_0/2} < (32+\bar
 C_\beta)^{-1}$) and restrict further the parameter $a$, so that $a
 \leq K_1/2$. Then, as $E^n_\beta(t,at) \leq 4m_0$ for $t \in [0,T]$,
 by the definition of $T$ we have
\begin{multline}
  \label{eq:di7}
  \td{}{t} \sum_{p=p_0}^n m_{\beta p} \frac{(at)^p}{p!}
  \leq - \frac{1}{4} K_1 I^n_{\beta,\beta} (t,at)
  + \frac1t (K_1+a) C_{p_0}
  \\
  \leq 
  - \frac{1}{t}  \left(
    \frac{K_1}{4a} (E^n _\beta (t,at) - m_0)
    - (K_1+a) C_{p_0} \right)
\end{multline}
where for the last inequality we have used that 
\[
I^n _{\beta,\beta} (t,at) \geq \frac{(E^n _\beta (t,at) - m_0)}{at}.
\]
We make the additional restriction that $a < m_0 / (6C_{p_0})$, which
together with $a < K_1/2$ implies that
\begin{equation*}
  \frac{K_1}{4a} m_0 > (K_1+a) C_{p_0}.
\end{equation*}
Then, whenever $E^n_\beta (t,at) \geq 2 m_0$, 
\begin{equation}
  \label{eq:di8}
  \td{}{t} \sum_{p=p_0}^n m_{\beta p} \frac{(at)^p}{p!}
  \leq 0
\end{equation}
for any time $t \in [0,T]$ for which $E^n_\beta (t,at) \geq 2 m_0$
holds. This is true in particular when $\sum_{p=p_0}^n m_{\beta p}
\frac{(at)^p}{p!} \geq 2m_0$. We deduce that
\begin{equation}
  \label{eq:large-moments-bound}
  \sum_{p=p_0}^n m_{\beta p} \frac{(at)^p}{p!}
  \leq
  2m_0 
  \quad \text{ for } t \in [0, T].
\end{equation}
In order to finish the argument we need to bound the initial part of
the full sum (from $p=0$ to $p_0-1$.) Indeed, we note that from
\eqref{eq:small-moments},
\begin{equation}
  \label{eq:small-moments-2}
  \sum_{p=0}^{p_0-1} m_{\beta p} \frac{(at)^p}{p!}
  \leq
  m_0 + aC_{p_0}
  \quad \text{ for } t \in [0, T],
\end{equation}
so, recalling that $6 a C_{p_0} < m_0$ and using
\eqref{eq:large-moments-bound} and \eqref{eq:small-moments-2}
\begin{equation*}
  E^n_\beta (t,at) =
  \sum_{p=0}^{p_0-1} m_{\beta p} \frac{(at)^p}{p!}
  + \sum_{p=p_0}^n m_{\beta p} \frac{(at)^p}{p!}
  \leq
  3 m_0 + a C_{p_0} 
  \leq \frac{19}{6} m_0 
\end{equation*}
for $t \in [0, T]$, uniformly in $n$. 
%
Finally, gathering all conditions imposed along the proof on the
parameter $a$, we choose
\begin{equation}\label{choose-a}
 a:=
\min \left\{1, \frac{K_1}{2}, \frac{m_0}{6C_{p_0}} \right\}
\end{equation}
independently of $n$, where $K_1$ was defined in \eqref{eq:di0.13.2} and
$C_{p_0}$ in \eqref{eq:small-moments}.  We conclude, from the
definition of $T$, that $T=1$ for all $n$.  Sending
$n\rightarrow\infty$, Theorem \ref{thm:exp-moment-creation} follows.

\bigskip

In the general case $\beta \in (0,2]$, since $K_2$ in
\eqref{eq:di0.13} is not zero, equation \eqref{eq:di4.5} has an extra
term in the right hand side, namely
\begin{equation*}
  \label{eq:di4.5-K2}
  \td{}{t} m_{\beta p}
  \leq
  2 \gamma_{\beta p/2} S_{\beta,p}
  - K_1 m_{\beta(p+1)} 
+ K_2  m_{\beta p}
  \quad \text{ for } t \geq 0, \ p \geq p_0.
\end{equation*}
In this case using again that $E^n _\beta (t,at) \leq 4m_0$ on
$[0,T]$, \eqref{eq:di7} is now modified 
as 
\begin{multline}
  \label{eq:di7bis}
  \td{}{t} \sum_{p=p_0}^n m_{\beta p} \frac{(at)^p}{p!}
  \leq - \frac{1}{4} K_1 I^n _{\beta,\beta} (t,at)
  + \frac1t (K_1+a) C_{p_0} + K_2 E^n _\beta (t,at)
  \\
  \leq
  -  \frac{1}{t} \left(
    \frac{K_1}{4a} (E^n _\beta (t,at) - m_0)
    - (K_1+a) C_{p_0} \right) + 4 K_2 m_0.
\end{multline}
Hence by tuning the constants as before, at any time $t \in [0,T]$ for
which $E^n_\beta (t,at) \geq 2 m_0$ we have 
\begin{equation*}
  \td{}{t} \sum_{p=p_0}^n m_{\beta p} \frac{(at)^p}{p!}
  \leq K_3
\end{equation*}
with $K_3 = 4 K_2 m_0$. The corresponding to equation
\eqref{eq:large-moments-bound} is then
\begin{equation*}
  \label{eq:large-moments-bound-K2}
  \sum_{p=p_0}^n m_{\beta p} \frac{(at)^p}{p!}
  \leq
  2m_0 + K_3 t
  \qquad t \in [0, T].
\end{equation*}
It follows as before that 
\begin{equation*}
  E^n_\beta (t,at) 
  \leq \frac{19}{6} m_0 +K_3 t\, ,  \qquad  t \in [0, T]\, ,
\end{equation*}
uniformly in $n$.  Then $T \geq m_0 / (2 K_3)$, where $K_3$ is a
constant which depends only on $b$, the hard potential exponent
$\beta$ and initial mass and energy.  In particular for the same rate
$a$ as in \eqref{choose-a} the conclusion follows since both $a$ and
$T$ are uniform in the index $n$ and the limit in $n$ can be performed
as well.
\end{proof}
\medskip

\begin{proof}[Proof of Theorem~\ref{thm:exp-moment-propag}]
  Consider again first the case $\beta \in (0,1]$, and $s \in
  [\beta,1]$ as in \eqref{eq:hyp-exp-init}, $a >0$ to be fixed later
  and $n \in \mathbb N$. Define $T > 0$ as
  \begin{equation*}
    \label{eq:def-T}
    T := \sup\big\{t > 0 \ \text{ s.t. } \ E^n _s(t,a) < 4 m_0 \big\}.
  \end{equation*}
  This definition is consistent since $E^n _s(0,a) \le E_s(0,a) < 4 m_0$
  for $a$ small enough thanks to the assumption
  \eqref{eq:hyp-exp-init} on the initial data, and the
  Lemma~\ref{lem:moment-creation} ensures that $T>0$ for each given
  $n$. We will show that, for $a$ chosen small enough, $T=+\infty$ for
  any $n$, thus proving the theorem.

  Choose an integer $p_0 > 2/s$, to be fixed later. Starting again
  from Lemma~\ref{lem:estim-poly} (inequality~\eqref{eq:di1} with the
  choice of constants \eqref{eq:di0.13.2}), we have
  \begin{equation}
    \label{eq:di4.5bis}
    \td{}{t} m_{sp}
    \leq
    2 \gamma_{sp/2} S_{s,p}
    - K_1 m_{sp+\beta} 
    \quad \text{ for } t \geq 0, \ p \geq p_0,
  \end{equation}
  with $S_{s,p}$ given by \eqref{eq:def-S_p} and $K_1$ given by
  \eqref{eq:di0.13.2}, independent of $p$ with $p \ge 0$.
  Also, from Lemma \ref{lem:moment-creation}
  (inequality~\eqref{eq:moment-propag}) we know that there exists a
  constant $C_{s,p_0} > 0$ (depending on $s$, $p_0$) such that
  \begin{equation}
    \label{eq:small-moments-bis}
    \sum_{p=0}^{p_0} m_{sp}  \leq C_{s,p_0}
    \quad \text{ for all } t \in [0,T].
  \end{equation}
  Taking any $a < \min\{1, a_0\}$, we have
  \begin{multline*}
    \td{}{t} \sum_{p=p_0}^n m_{sp} \frac{a^p}{p!}
    \leq \sum_{p=p_0}^n \frac{a^p}{p!}  \left( 2
      \gamma_{sp/2} S_{s,p} - K_1 m_{sp+\beta} 
\right) 
    \\
    \leq 2 \sum_{p=p_0}^n \frac{a^p}{p!} \gamma_{sp/2} S_{s.p} - K_1 I^n_{s,\beta}(t,a) 
+ K_1 \sum_{p=0}^{p_0-1} m_{sp+\beta}
    \,\frac{a^{p}}{p!}
    \\
    \leq 2 \sum_{p=p_0}^n \frac{a^p}{p!} \gamma_{sp/2} S_{s,p} 
    - K_1 I^n _{s,\beta}(t,a)
+ K_1 C_{s,p_0},
  \end{multline*}
  where we have used $a < 1$ and \eqref{eq:small-moments-bis} in the
  last step.  Hence, from Lemma~\ref{lem:convolution}
  (inequality~\eqref{eq:convolution}) we obtain
  \begin{equation}
    \label{eq:di6}
    \td{}{t} \sum_{p=p_0}^n m_{sp} \frac{a^p}{p!}
    \leq I^n_{s,\beta}(t,a) \Big[
    4 {\gamma}_{s p{_0}} E^n_s(t,a) - K_1    \Big]
    + K_1 C_{s,p_0},
  \end{equation}
  where, as in the previous proof, we also choose $p_0$ such that $16
  \gamma_{s p_0/2} m_0 \leq (1/2) K_1$. 
  Then, as $E^n_s(t,a) \leq 4m_0$ for $t \in
  [0,T]$ by definition of $T$ we have
  \begin{equation*}
    \td{}{t} \sum_{p=p_0}^n m_{sp} \frac{a^p}{p!}
    \leq - \frac{1}{2} K_1 I^n_{s,\beta}(t,a)
    + K_1 C_{s,p_0}
    \leq
    -  \frac{K_1}{2a} E^n_s(t,a) + K_1 \left( \frac{m_0}{2a} + e^a \right)
    + K_1 C_{s,p_0},
  \end{equation*}
  where for the last inequality we have used that 
\begin{multline*}
I^n_{s,\beta}(t,a) \geq \int_{|v| \ge 1} \left( \sum_{p=1} ^n |v|^{sp+\beta}
  \frac{a^p}{p!}  \right) f \dd v  \ge \int_{|v| \ge 1} \left(
  \sum_{p=1} ^n |v|^{sp} \frac{a^p}{p!} \right) f \dd v \\ \ge \int_{\R^d} \left(
  \sum_{p=1} ^n |v|^{sp} \frac{a^p}{p!} \right) f \dd v - e^a \,
\int_{\R^d} f \dd v \ge 
\frac{(E^n_s(t,a) - m_0)}{a} - e^a,
\end{multline*}
  so that
  \begin{equation}\label{eq:di6.1}
  \td{}{t} \sum_{p=p_0}^n m_{sp} \frac{a^p}{p!} \le -
  \frac{K_1}{2a} E^n_s(t,a) + K_1 \left( \frac{m_0}{2a} + e^a \right) + K_1 C_{s,p_0}.   
  \end{equation}
  Next, recalling estimate \eqref{ee1} in the proof of
  Lemma~\ref{lem:moment-creation}
  \[
  \td{}{t} m_{sp} \le C' m_{sp} 
  \]
 valid for any $p \in \mathbb N$ and constant $C'$ depending only on $s$, $p$, initial mass and energy. 
 Summing in $p$, from 
$0$ to $p_0-1$, and using estimate \eqref{eq:moment-propag} 
 we obtain
  \begin{equation*}
  \td{}{t} E^n_s(t,a) \le -
  \frac{K_1}{2a} E^n_s(t,a) + K_1 \left( \frac{m_0}{2a} + e^a \right) + (K_1 + C')C_{s,p_0}  \, . 
  \end{equation*}
 This implies, by a maximum principle argument for ODEs,
  that the bound
  \[
  E^n_s(t,a) \le  m_0 + 2 a\left[ \left( 1 + \frac{C'}{K_1}\right)
    C_{s,p_0}+e^a \right]
  \]
  holds uniformly for $t\in[0,T]$, as the parameters in the right hand
  side are uniform in time.  Choosing $a$ small enough such that
  \[
m_0 + 2 a\left[ \left( 1 + \frac{C'}{K_1}\right)
    C_{s,p_0}+e^a \right]  <  4m_0\, ,
  \]
or equivalently
 \[
 a <  \min\left\{1, a_0, \frac{K_1}{2} ,\frac{3m_0}{2 \left[ \left( 1 + \frac{C'}{K_1}\right)
    C_{s,p_0}+e^a \right]} \right\} ,
  \]
  where $K_1$ was defined in \eqref{eq:di0.13.2} and $C_{s,p_0}$ in
  \eqref{eq:small-moments-bis}, proves by definition of $T$ that
  $T=+\infty$ for any $n$.  Passing to the limit $n \to +\infty$
  concludes the proof.

\bigskip

In the general case $\beta\in (0,2]$, again as in the previous proof it
follows that equation \eqref{eq:di4.5bis} has the extra positive term in
the right hand side $ K_2 m_{sp}$.  The corresponding equation to
\eqref{eq:di6} is now
\begin{multline}
    \label{eq:di6-K2}
    \td{}{t} \sum_{p=p_0}^n m_{sp} \frac{a^p}{p!}
    \leq I^n_{s,\beta}(t,a) \Big[
    4 {\gamma}_{s p{_0}} E^n_s(t,a)
    - K_1 \Big] \\ + K_2 E^n_s(t,a)
    + (K_1 +K_2) C_{s,p_0} 
  \end{multline} 
 and consequently, arguing as before we get 
\begin{equation*}
  \td{}{t} \sum_{p=p_0}^n m_{sp} \frac{a^p}{p!}
  \leq
  \left(K_2 -  \frac{K_1}{2a}\right) E^n_s(t,a) + K_1 \left(
    \frac{m_0}{2a} + e^a \right) 
  + (K_1 +K_2) C_{s,p_0}.
  \end{equation*}
  In particular, making the additional restriction that $a <
  K_1/(4K_2)$ we obtain 
  the bound
  \[
  E^n_s(t,a) \le 2 m_0 + 4 a\left[ \left( 1 + \frac{K_2}{K_1} + \frac{C'}{K_1}\right)
    C_{s,p_0}+e^a \right] 
  \]
  uniformly for  $t\in[0,T]$, where now $a$ is chosen so that
  \[
  a < \min\left\{1, a_0,   \frac{K_1}{8K_2}, 
  \frac{m_0}{2\left[ \left( 1 + \frac{K_2}{K_1} + \frac{C'}{K_1}\right)
    C_{s,p_0}+e^a \right]} \right\}  ,
  \]
  with $K_1$ and $K_2$ given in \eqref{eq:di0.13}, with $p_0$ such
  that $\gamma_{s p_0/2} < (32+ 2^{1-\beta})^{-1}$.  The proof is then
  completed as in the case $\beta\in (0,1]$ above. 
\end{proof}

\appendix 

\section{Some technical tools on moments} 

We collect here two technical calculations from previous works.

\begin{lem}[Lemma 2 in \cite{BGP04}]\label{lem:app1} 
  Assume that $p>1$, and let $k_p$ denote the integer part of
  $(p+1)/2$. Then for all $x,y >0$ the following inequalities hold
  \begin{equation*}
    \sum_{k=1} ^{k_p-1}
    {p \choose k} \left( x^k y^{p-k} + x^{p-k} y^k \right)
    \le
    (x+y)^p - x^p - y^p
    \le
    \sum_{k=1}^{k_p}
    {p \choose k} \left( x^k y^{p-k} + x^{p-k} y^k \right).
    \end{equation*}
\end{lem}

\begin{lem}[Lemma 2 in \cite{GPV09}]\label{lem:app2}
  The energy-conserving solutions to the Boltzmann equation
  \eqref{eq:Boltzmann} on $[0,+\infty)$ with initial data $f^0 \in
  L^1(1+|v|^2)$ satisfy 
  \begin{equation*}
    \forall \, t \ge 0, \ \forall \,v \in \R^d, \quad 
    \int_{\R^d} f(t,v_*) |v-v_*|^s \, {\rm d}v_* \ge c_s \, \int_{\R^d} f^0(v_*) |v-v_*|^s \, {\rm d}v_*
  \end{equation*}
  for any $s \in (0,1]$ and some constant $c_s>0$ depending on
  $s$. This implies that
\begin{equation*}
    \forall \, t \ge 0, \ \forall \, v \in \R^d, \quad 
    \int_{\R^d} f(t,v_*) |v-v_*|^s \, {\rm d}v_* \ge C_{f^0,s}
    \, (1+|v|^s)
  \end{equation*}
  for any $s \in (0,1]$ and some constant $C_{f^0,s}>0$ depending on $s$
  and the initial data $f^0$. 
\end{lem}

\signra 
\signjc
\signig
\signcm
\end{document}